\documentclass{amsart}
\usepackage[utf8]{inputenc}
\usepackage{latexsym, amssymb,amsthm,amsmath,verbatim,xcolor,enumerate}

\usepackage{amsmath}

\begin{document}

\newtheorem{definition}{Definition}[section]

\newtheorem{proposition}[definition]{Proposition}

\newtheorem{claim}{Claim}[definition]

\newtheorem{prob}[definition]{Problem}

\newtheorem{lemma}[definition]{Lemma}

\newtheorem{corollary}[definition]{Corollary}
\newtheorem{theorem}[definition]{Theorem}

\newtheorem{notation}[definition]{Notation}
\newtheorem{exercise}[definition]{Exercise}

\newtheorem{example}[definition]{Example}

\newtheorem{thm}[definition]{Theorem}

\newtheorem{question}[definition]{Question}

\newtheorem{fact}[definition]{Fact}

\theoremstyle{definition}

\newtheorem{remark}[definition]{Remark}

\newcommand{\seq}[1]{{\left\langle #1 \right\rangle}}
\newcommand{\fc}{{\mathfrak{c}}}
\newcommand{\mR}{\mathcal{R}}
\newcommand{\RR}{\mathbb{R}}
\newcommand{\QQ}{\mathbb{Q}}
\newcommand{\NN}{\mathbb{N}}
\newcommand{\ZZ}{\mathbb{Z}}
\newcommand{\mA}{\mathcal{A}}
\newcommand{\mV}{\mathcal{V}}
\newcommand{\mH}{\mathcal{H}}
\newcommand{\mU}{\mathcal{V}}
\newcommand{\mP}{\mathcal{P}}
\newcommand{\mD}{\mathcal{D}}
\newcommand{\mB}{\mathcal{B}}
\newcommand{\C}{\mathrm{C}}
\newcommand{\mC}{\mathcal{C}}
\newcommand{\mK}{\mathcal{K}}
\newcommand{\mO}{\mathcal{O}}
\renewcommand{\O}{\mathcal{O}}
\newcommand{\mM}{\mathcal{M}}
\newcommand{\mF}{\mathcal{F}}
\newcommand{\D}{\mathrm{D}}
\newcommand{\0}{\mathrm{o}}
\newcommand{\OD}{\mathrm{OD}}
\newcommand{\Do}{\D_\mathrm{o}}
\newcommand{\sone}{\mathsf{S}_1}
\newcommand{\gone}{\mathsf{G}_1}
\newcommand{\gtwo}{\mathsf{G}_2}
\newcommand{\gn}[1]{\mathsf{G}_{#1}}
\newcommand{\sn}[1]{\mathsf{S}_{#1}}
\newcommand{\bmfin}{\mathsf{BM}_\mathrm{fin}}
\newcommand{\bmc}{\mathsf{BM}_\mathrm{\w}}
\newcommand{\bm}{\mathsf{BM}}
\newcommand{\sfin}{\mathsf{S}_\mathrm{fin}}
\newcommand{\gf}{\mathsf{G}_f}
\newcommand{\Em}{\longrightarrow}
\newcommand{\Def}[1]{{\bf #1}}
\newcommand{\menos}{{\setminus}}
\newcommand{\w}{{\omega}}
\newcommand{\1}{\textsc{Alice}}
\newcommand{\2}{\textsc{Bob}}
\newcommand{\lev}{\mathrm{Lev}}
\newcommand{\cov}{\mathrm{cov}}
\newcommand{\fb}{\mathrm{fb}}
\newcommand{\dom}{\mathrm{dom}}
\newcommand{\id}{\mathrm{id}}

\newcommand{\first}{\textit{first inning}}
\newcommand{\second}{\textit{second inning}}
\newcommand{\third}{\textit{third inning}}

\title{Some variations of the Banach-Mazur game}
\author{Leandro F. Aurichi$^1$}\thanks{$^1$Supported by FAPESP 2019/22344-0 and by National Group for Algebraic and Geometric Structures, and their Applications (GNSAGA-INdAM)}
\address{Instituto de Ci\^encias Matem\'aticas e de Computa\c c\~ao,
Universidade de S\~ao Paulo, Caixa Postal 668,
S\~ao Carlos, SP, 13560-970, Brazil}
\email{aurichi@icmc.usp.br}

\author{Maddalena Bonanzinga}
\address{Dipartimento di Scienze Matematiche e Informatiche, Scienze Fisiche e Scienze del la Terra, Universit\'a di Messina – Viale F. Stagno d'Alcontres 31, 98166 Messina, Italy}
\email{mbonanzinga@unime.it}

\author{Gabriel Andre Asmat Medina$^3$}\thanks{$^3$Supported by CAPES (88882.328757/2019-01)}
\address{Instituto de Ci\^encias Matem\'aticas e de Computa\c c\~ao,
Universidade de S\~ao Paulo, Caixa Postal 668,
S\~ao Carlos, SP, 13560-970, Brazil}
\email{andre\_asmat@usp.br}

\date{}

\begin{abstract}

The classical Banach-Mazur game is directly related to the Baire
property and the property of being a productively Baire space. In this
paper, we discuss two variations of this classic game that are even
more related to these properties.


\end{abstract}

\maketitle

\section{Introduction}

The famous Banach–Mazur game was proposed in 1935 by Stanislaw Mazur and recorded in the Scottish Book (\cite{Mauldin1981}, Problem 43). Let $(X, \tau)$ be a topological space. The Banach-Mazur game $\bm(X)$ played on $(X, \tau)$ is played between two players, \1 and \2, who, alternately, select non-empty open subsets of $X$. \1 goes first and chooses a non-empty open subset $A_{0}$ of $X$. \2 must respond by selecting a non-empty open subset $B_{0}\subseteq A_{0}$. Following this, \1 must select another non-empty open subset $A_{1}\subseteq B_{0} \subseteq A_{0}$ and in turn \2 must again respond by selecting a non-empty open subset $B_{1}\subseteq A_{1}\subseteq B_{0} \subseteq A_{0}$. In general, \1 selects any non-empty open subset $A_{n}$ of the last move $B_{n-1}$ of \2 and the latter player answers by choosing a non-empty open subset $B_{n}$ of the set $A_{n}$, just chosen by \1. Acting in this away, the players \1 and \2 produce a sequence of non-empty open sets 
$$  A_{0}\supseteq B_{0} \supseteq A_{1}\supseteq B_{1} \supseteq \cdots \supseteq A_{n} \supseteq B_{n} \supseteq \cdots $$
We shall declare that \2 wins a play of the Banach-Mazur game $\bm(X)$ if $ \bigcap_{n\in \omega} B_{n} = \bigcap_{n\in \omega} A_{n}\not=\emptyset$. Otherwise, \1 is said to be the winner of this play. 

Remember that a topological space is a Baire space provided countable intersections of dense open subsets are dense. Baire spaces can be characterized via the Banach–Mazur game. In fact, Oxtoby showed that a topological space $X$ is Baire if and only if \1 does not have a winning strategy in $\bm(X)$ (\cite{Kechris1995}, Theorem 8.11). Products of Baire spaces are not always Baire (\cite{Fleissner1978}, Example 1 and Example 4). A Baire space $X$ is productively Baire if $X\times Y$ is Baire, for each Baire space $Y$. Another application of the Banach-Mazur game to Baire spaces, is the following: if \2 has a winning strategy in $\bm(X)$ then $X$ is productively Baire. The reciprocal of this result does not hold. For example a Bernstein subset of the real line is productively Baire but the Banach-Mazur game is undertermined, that is, neither \1 nor \2 have a winning strategy in the Banach-Mazur game on such a space.

That is why the following question is natural: Is there a game-theoretical characterization for the property of being productively Baire?

In this paper we show some other game-theoretical conditions on the space to be productively Baire. For this we introduce two variations of the Banach-Mazur game.

This paper is organized as follows. In Section 2 we present the game $\bmfin$, our first variation of the Banach-Mazur game, where the second player may choose finitely many open sets each inning. With the game $\bmfin$ we present a game-theoretical condition on the space to be productively Baire, which is more general than the original one using the Banach-Mazur game.

In Section 3 we present another variation of the Banach-Mazur game, the game $\bmc$. This game is similar to the previous one, with the difference that the second player chooses a countable family of non-empty open sets. 
It is worth mentioning that both of these game, as the original Banach-Mazur game, provide a  characterization for the Baire property. 

Finally, Section 4 is dedicated to present an example, assuming the Continuum Hypothesis, that shows the games $\bmfin$ and $\bmc$ are not equivalent. We also present some open questions.

\section{The game $\bmfin(X)$}

In this section we introduce our first variation of the Banach-Mazur game. The difference with the classic Banach-Mazur game is that in this version $\2$ has the advantage of choosing a finite number of non-empty open sets instead of just one.

\begin{definition}
	Given a topological space $X$, we define the game $\bmfin(X)$ played as follows: \1 plays $A_0$ a non-empty open set. Then \2 plays $\mB_0$ a finite collection of non-empty open subsets of $A_0$. Then, for each $B \in \mB_0$, \1 plays $A_B \subset B$ a non-empty open set. Let $\mA_1 = \{A_B: B \in \mB_0\}$. Then \2 plays $\mB_1$ a finite collection of non-empty open subsets of $\bigcup \mA_1$ and so on. For each $n \in \w$, let $B_n = \bigcup \mB_n$. \2 is declared the winner if $\bigcap_{n \in \w} B_n \neq \emptyset$. \1 is declared the winner otherwise. 
\end{definition}

Sometimes it is easier to see this game with a minor change in the
rules for \2. In each inning, instead of picking finitely many
non-empty open subsets of  $\bigcup \mA_n$, Bob can pick finitely many
open subsets of each open set played by \1 - including picking none
for some of them. These two versions of the rules are easily seen to
be equivalent.

We present the following technical lemma which tells us that we can assume that each move of \2 can be formed by a finite family of non-empty open sets  pairwise disjoint.

\begin{lemma}\label{lema disjoint open}
	Let $X$ be a Hausdorff space without isolated points. We may suppose that each $\mB_n$ as in the definition of $\bmfin(X)$ is made of pairwise disjoint sets.
\end{lemma}

\begin{proof}	
	We will show that if $\sigma$ is a winning strategy for \2 then there exists another winning strategy $\hat{\sigma}$ for \2 such that every move from \2 is 	formed by non-empty pairwise disjoint open sets.
	
	Indeed, in the first inning, \1 plays $A_{0}$. Next \2 responds $\sigma(\langle A_{0} \rangle)=\mathcal{B}_{0}=\{B^{0}_{0}, \cdots,B^{0}_{n_{0}}\}$ for some $n_{0}\in\omega$. So let $x^{0}_{i}\in B^{0}_{i}$ for each $i\in \{0,\cdots, n_{0}\}$ such that $x^{0}_{i}\not=x^{0}_{j}$ if $i\not=j$. Note that this is possible because $X$ has no isolated points. Now, since $X$ is Hausdorff, there is a pairwise disjoint family of non-empty open sets $\hat{\mathcal{B}_{0}}=\{\hat{B^{0}_{0}}, \cdots, \hat{B^{0}_{n_{0}}} \}$ such that $\hat{B^{0}_{i}} \subseteq B^{0}_{i} $ for each $i\in \{0,\cdots, n_{0}\}$. So define $\hat{\sigma}(\langle A_{0} \rangle)=\hat{\mathcal{B}_{0}}$. 
	
	Next, in the second inning, \1 plays $\mathcal{A}_{1}=\{A_{B}:B\in \hat{\mathcal{B}_{0}}\}$, then \2 responds $\sigma(\langle A_{0}, \mathcal{A}_{1}\rangle)=\mathcal{B}_{1}=\bigcup_{A\in \mathcal{A}_{1}} \mathcal{F}^{1}_{A}$, where each $\mathcal{F}^{1}_{A}$ is a finite family of non-empty open subsets of $A\in\mathcal{A}_{1}$. Now for each $\mathcal{F}^{1}_{A}=\{B^{1}_{0}, \cdots, B^{1}_{n_{1}}\}$, we apply the above argument, that is, there is a family $\hat{\mathcal{F}^{1}_{A}} = \{ \hat{B^{1}_{0}}, \cdots, \hat{B^{1}_{n_{1}}}\}$ of pairwise disjoint non-empty open sets such that $\hat{B^{1}_{i}}\subseteq B^{1}_{i}$ for each $i\in\{0, \cdots, n_{1}\}$, so define $\hat{\sigma}(\langle A_{0}, \mathcal{A}_{1} \rangle)= \bigcup_{A\in\mathcal{A}_{1}} \hat{\mathcal{F}^{1}_{A}}$ and so on.
	
	As $\sigma$ is a winning strategy then there exists $x\in \bigcap_{n \in \w} \bigcup\sigma(\langle A_{0}, \mathcal{A}_{1}, \cdots, \mathcal{A}_{n} \rangle)=\bigcap_{n \in \w} \bigcup\mathcal{A}_{n}=\bigcap_{n \in \w} \bigcup\hat{\sigma}(\langle A_{0}, \mathcal{A}_{1}, \cdots, \mathcal{A}_{n} \rangle)$. Therefore $\hat{\sigma}$ is a winning strategy for \2 in $\bmfin(X)$.
\end{proof}






Now we can generalize Oxtoby's theorem.

\begin{proposition}
  Let $X$ be a Hausdorff space without isolated points. If \2 has a winning strategy on $\bmfin(X)$, then $X$ is productively Baire. 
\end{proposition}

\begin{proof}
	Let $\sigma$ be a winning strategy for \2 in the $\bmfin(X)$ game. Suppose that $X$ is not productively Baire, that is, there exists a Baire space, call it $Y$, such that $X\times Y$ is not a Baire space.

	Suppose that $\rho$ is a winning strategy for \1 in the $\bm(X \times Y)$ game. We will define a winning strategy $f$ for \1 in the $\bm(Y)$ game. Note that this is enough. Also, notice that we can suppose that $\rho$ always answer with basic open sets in the product. 
	
	In the first inning, in $\bm(X\times Y)$, \1 plays $A_0 \times B_0 = \rho(\seq{})$. Then, in $\bm(Y)$, \1 plays $f(\seq{}) = B_0$. Next \2 responds $W_0 \subset B_0$ a non-empty open set. We now will proceed to define $f(\seq{W_0})$. In $\bmfin(X)$, \1 plays $A_{0}$ and \2 responds $\{V_0, ..., V_n\} = \sigma(\seq{A_0})$. By Lemma \ref{lema disjoint open}, we may suppose that $V_i \cap V_j = \emptyset$ if $i \neq j$.

	In the second inning, in $\bm(X\times Y)$, \1 could play in the following $n + 1$ ways:    
	
	\[A_1^0 \times C_0 = \rho(\seq{V_0 \times W_0})\]
	\[A_1^1 \times C_1 = \rho(\seq{V_1 \times C_0})\]
	\[\vdots\]
	\[A_1^n \times C_n = \rho(\seq{V_n \times C_{n - 1}})\]
	
	Note this is valid since $C_0 \subset W_0 \subset B_0$ and each $C_k \subset C_{k - 1} \subset B_0$ for $0<k\leq n$, also $A^{k}_{1}\subset V_{k}$ for each $0\leq k \leq n$. Let $B_1 = C_{n}$. Finally, in $\bm(Y)$, \1 plays $f(\seq{W_0}) = B_1$. Next \2 responds $W_1 \subset B_1$ a non-empty open set and let us define $f(\seq{W_0, W_1})$. Now, in $\bmfin(X)$, \1 plays $\{A^{0}_{1}, \cdots, A^{n}_{1}\}$ next \2 responds 
	\[\{U_0, ..., U_k\} = \sigma(\seq{A_0, \{A_1^0, ..., A_1^n\}}).\]

	By Lemma \ref{lema disjoint open} we may suppose that $U_0, ..., U_k$ and $A_1^0, ..., A_1^n$ are pairwise disjoint sets. Then for each $i = 0, ..., k$, there is only one $g(i)\in \{0, \cdots, n  \}$ such that \[U_i \subset A_1^{g(i)}.\]

	In the third inning, in $\bm(X\times Y)$, \1 could play in the following $k + 1$ ways:    
	
	\[A_2^0 \times D_0 = \rho(\seq{V_{g(0)} \times C_{g(0) - 1}, U_0 \times W_1})\]
	\[A_2^1 \times D_1 = \rho(\seq{V_{g(1)} \times C_{g(1) - 1}, U_1 \times D_0})\]
	\[\vdots\]
	\[A_2^k \times D_k = \rho(\seq{V_{g(k)} \times C_{g(k) - 1}, U_k \times D_{k - 1}}).\]  
	
	On the previous equations, use $C_{-1} = W_0$ if neccessary. As before, $D_i \subset D_{i -1} \subset W_1$. Let $B_{2}=D_{k}$ and, finally, in $\bm(Y)$, \1 plays $f(\seq{W_0, W_1})=D_k$. Next in $\bm (Y)$, \2 responds $W_2 \subset B_2$ a non-empty open set. Now, in $\bmfin(X)$, \1 can play $\{A^{0}_{2}, \cdots, A^{k}_{2}\}$ next \2 responds $\sigma(\langle A_{0},\{A_1^0, ..., A_1^n\} , \{A^{0}_{2}, \cdots, A^{k}_{2}\}\rangle)$. The remaining construction of $f$ is in similar fashion.
	
	Let us prove now that $f$ is winning for \1 in $\bm(Y)$. So fix a play $$\seq{B_0, W_0, B_1, W_1, ..., B_n, W_n, ...}$$ following $f$ (with notation similar to what we presented in the construction of $f$). We need to show that $\bigcap_{n \in \w} B_n = \emptyset$.  Suppose not and let $y \in \bigcap_{n \in \w} B_n$. 
	
	Let $\mA_n = \{A_n^0, ..., A_n^{k_n}\}$. Since $\sigma$ is winning for \2 in $\bmfin(X)$, there is an
	\[x \in \bigcap_{n \in \w} \bigcup \mA_n.\]
	Since we are supposing that each $\mA_n$ is made by mutually disjoint open sets, there is only one $A_n$ in each $\mA_n$ such that $x \in A_n$.

	In particular, in $\bm(X\times Y)$, in the first inning, $\langle x, y\rangle\in A_{0}\times B_{0}=\rho(\seq{})$. In the second inning, there is only one $A^{n_{0}}_{1}\in \mathcal{A}_{1}$ such that $x\in A^{n_{0}}_{1}$. Then $\langle x, y\rangle\in A^{n_{0}}_{1}\times C_{n_{0}}=\rho(\langle V_{n_{0}}\times C_{n_{0}-1}  \rangle)$. In the third inning, as $x\in \bigcup\mathcal{A}_{2}$, there is only one $A^{n_{1}}_{2}\in\mathcal{A}_{2}$ such that $x\in A^{n_{1}}_{2}$. Then $\langle x, y\rangle\in A^{n_{1}}_{2}\times D_{n_{1}}=\rho(\langle
	V_{n_{0}}\times C_{n_{0}-1}, U_{n_{1}}\times D_{n_{1}-1}  \rangle)$. Procceding like this, we can find a play of $\bm(X \times Y)$ that has $\seq{x, y}$ in all its movements but it is compatible with $\rho$, which is a contradiction.

\end{proof}


We will show in Section 3 that the game $\bmfin$ is equivalent to the Banach-Mazur game in the point of view of \1 (i.e she has a winning strategy in one of the games if, and only if, she has one in the other). But this is not the case for \2. We will end this section showing that a Bernstein set is a witness for this.

Remember that a subset $X\subseteq \RR$ is a Bernstein set if both $X$ and $\mathbb{R}\setminus X$ meets every uncountable closed subset of the real line. A Bernstein set has no isolated points and is dense in $\RR$. To see that \2 has no winning strategy for the $\bm$ game in a Bernstein set just recall the following:

\begin{proposition}
	If $X \subset \RR$ has no isolated points and \2 has a winning strategy on $\bm(X)$, then $X$ contains a Cantor set.
\end{proposition}

The previous result shows that \2 does not have a winning strategy for the $\bm(X)$ game played on a Bernstein set.


On the other hand, \2 has a winning strategy for $\bmfin$ in a Bernstein set

\begin{proposition}\label{2bmfinberstein}
  Suppose that $X \subset \RR$ is a Bernstein set. Then \2 has a winning strategy on $\bmfin(X)$. 
\end{proposition}

\begin{proof}
  For every open set $V \subset X$, let $V^*$ be an open set in $\RR$ such that $V = V^* \cap X$. On each inning, \2 will split every open set played by \1 in two pieces, therefore, we will enumerate all the open sets played with sequences in $2^{<\w}$. On the inning of number $n$, let
  \[\{A_s: s \in 2^n\}\]
  be the open sets played by \1. For each $s \in 2^n$, let $B_{s \smallfrown 0}$ and $B_{s \smallfrown 1}$ be two open sets such that:
  \begin{itemize}
  \item $\emptyset \neq B_{s \smallfrown i} \subset \overline{B_{s \smallfrown i}^*}^\RR \subset A_s^*$;
  \item $\overline{B_{s \smallfrown 0}^*}^\RR \cap \overline{B_{s \smallfrown 1}^*}^\RR = \emptyset$;
  \item the diameter of $B_{s \smallfrown i}$ is less than $\frac{1}{n + 1}$.
  \end{itemize}

  Then, in the next inning, we define $A_s$ as the open set given by \1 such that $A_s \subset B_s$.

  For each $n > 0$, let $B_n = \bigcup\{B_s: l(s) = n\}$. Note that, by the way we defined it, $C = \bigcap_{n \in \w} \overline{B_n^*}$ is a Cantor set. Therefore, $C \cap X \neq \emptyset$ and thus \2 wins the game. 
\end{proof}

The two previous results show that $\bm$ and $\bmfin$ are not equivalent for \2. 


\section{The game $\bmc(X)$}

In this part we present the game $\bmc$. The difference from the game $\bmfin$ is basically that in each inning \2 plays a countable number of non-empty open subsets contained in the union of the last move of \1. Instead of just finitely many as the previous game. 

\begin{definition}
  Given a topological space $X$, we define the game $\bmc(X)$ played as follows: \1 plays $A_0$ a non-empty open set. Then \2 plays $\mB_0$ a countable collection of non-empty open subsets of $A_0$. Then, for each $B \in \mB_0$, \1 plays $A_B \subset B$ a non-empty open set. Let $\mA_1 = \{A_B: B \in \mB_0\}$. Then \2 plays $\mB_1$ a countable collection of non-empty open subsets of $\bigcup \mA_1$ and so on. For each $n \in \w$, let $B_n = \bigcup \mB_n$. \2 is declared the winner if $\bigcap_{n \in \w} B_n \neq \emptyset$. \1 is declared the winner otherwise. 
\end{definition}

The following results show that the games $\bm, \bmfin$  and $\bmc$ are all equivalent for \1.

\begin{proposition}\label{n1B}
If \1 does not have a winning strategy on $\bmc(X)$, then $X$ is a Baire space.  
\end{proposition}

\begin{proof}
We will show that if $X$ is not a Baire space then \1 has a winning strategy $\delta$ on $\bmc(X)$. 

Since, $X$ is not Baire, there are a sequence $\langle D_{n} : n\in\omega\rangle$ of open dense subsets of $X$ and a non-empty open subset $A$ of $X$ such that $\bigcap_{n \in \w}D_{n}\cap A=\emptyset$. Then, in the first inning, \textcolor{black}{\1 plays $\delta(\seq{})=A_{0}=A$} and \2 responds $\mB_{0}=\{B^{0}_{n} : n\in\omega  \}$, a countable collection of non-empty open subsets of $A_{0}$. In the second inning, \textcolor{black}{\1 plays $\delta(\langle \mB_{0} \rangle)= \mA_{1}=\{D_{0}\cap B^{0}_{n} :n\in\omega\}$}. Note that this move is valid because $D_{0}$ is open dense. Next \2 responds $\mB_{1}=\{B^{1}_{n}:n\in\omega\}$ with $B^{1}_{n}\subseteq \bigcup\mA_{1}$ for each $n\in\omega$. In the third inning, as $D_{1}$ is dense, \textcolor{black}{\1 can play $\delta(\langle \mB_{0}, \mB_{1} \rangle)= \mA_{2}=\{D_{1}\cap B^{1}_{n} :n\in\omega\}$}. Next \2 responds $\mB_{2}=\{B^{2}_{n}:n\in\omega\}$ with $B^{2}_{n}\subseteq \bigcup\mA_{2}$ for each $n\in\omega$, and so on. The remaining construction of $\delta$ is in similar fashion.

Let us prove now that $\delta$ is winning for \1 in $\bmc(X)$. So fix a play $$\seq{A_0, \mB_0, \mA_1, \mB_1, ..., \mA_n, \mB_n, ...}$$ following $\delta$ (with notation similar to what we presented in the construction of $\delta$). We need to show that $\bigcap_{n \in \w} B_n = \emptyset$, where $B_n=\bigcup{\mB}_n$ for every $n\in\omega$. In fact, note that $B_{0}\subseteq A$ and for each $n>0$ we have $B_{n}\subseteq D_{n-1}$. Then $\bigcap_{n \in \w}B_{n}\subseteq \bigcap_{n \in \w}D_{n}\cap A=\emptyset$. Therefore $\delta$ is winning for \1.

\end{proof}

\begin{corollary}
	Let $X$ be a non-empty topological space. Then the following are equivalent:
	\begin{enumerate}
		\item $X$ is Baire
		
		\item \1 does not have a winning strategy on $\bm(X)$

		\item \1 does not have a winning strategy on $\bmfin(X)$
		
		\item \1 does not have a winning strategy on $\bmc(X)$
		
    \end{enumerate}
\end{corollary}

\begin{proof}
The implication $(1)\Rightarrow (2)$ is the classic Oxtoby's result (\cite{Kechris1995}, Theorem 8.11). Note that if $\1$ has a winning strategy on $\bmfin(X)$	then $\1$ has a winning strategy on $\bm(X)$, because the game $\bmfin$ is more difficult for \1. This proves $(2)\Rightarrow (3)$. 
Also, if $\1$ has a winning strategy on $\bm_\omega(X)$, then $\1$ has a winning strategy on $\bmfin(X)$, because the game $\bm_\omega$ is more difficult for \1. This proves $(3)\Rightarrow (4)$. Finally, by Proposition \ref{n1B}, we have $(4)\Rightarrow (1)$.

\end{proof}


We end this section showing that $\bm_\omega$ is determined for a broad class of spaces.

\begin{definition}
A family $\mB$ of non-empty open sets in a topological space will be called a \textbf{$\pi$-base} if every non-empty open set contains at least one member of $\mB$. A $\pi$-base $\mB$ is called \textbf{locally countable} if each member of $\mB$ contains countably many members of $\mB$.
\end{definition}

\begin{proposition}[\cite{Oxtoby1961}]
	The cartesian product of a Baire space $X$ and a Baire space $Y$ having countable $\pi$-base is a Baire space.
\end{proposition}

\begin{corollary}\label{cpbpb}
	If $X$ is a second countable Baire space then $X$ is productively Baire.
\end{corollary}


\begin{remark}\label{jogopibase}
Let $\mB$ be a $\pi$-base for the topology of the space $X$. Note that, in the $\bmc$ game on $X$, we can assume that, both \1 and \2 must necessarily choose elements from $\mB$ in their moves. This is also valid if we change $\bmc$ by $\bm$. This fact will be used freely from now on without mentioning it.
\end{remark}


\begin{theorem}
  If $X$ is a Baire space with a \textcolor{black}{locally} countable $\pi$-base, then \2 has a winning strategy on $\bmc(X)$.  
\end{theorem}

\begin{proof}
  First, we will define a strategy $\sigma$ for \2 in the $\bmc(X)$ game using the fact that $X$ has a \textcolor{black}{locally} countable $\pi$-base $\mB$. 
  
 Indeed, in the first inning, \textcolor{black}{\1 plays $A_0$} be any element from $\mB$ and let $\mA_0 = \{A_0\}$ next \textcolor{black}{\2 responds $\sigma(\langle A_{0}  \rangle) = \{B \in \mB: B \subset A_0\}$}, note that this is a valid move since $\mB$ is locally countable. In the second inning, \textcolor{black}{\1 plays $\mA_{1}=\{A_{B}\in\mB : B\in \sigma(\langle A_{0}  \rangle) \}$} and \textcolor{black}{\2 responds $\sigma(\langle A_{0}, \mA_{1} \rangle) = \bigcup_{A\in \mathcal{A}_{1}}\{B \in \mB: B \subset A   \}$}, note that this is a valid move since $\mA_{1}$ is countable and $\mB$ is locally countable. 
 Then, if $\seq{\mA_n: n \leq k}$ are the plays of \1, define \textcolor{black}{$\sigma(\seq{\mA_0, ..., \mA_n}) = \bigcup_{A\in\mA_{n}}\{B \in \mB: B\subset A \}$}. This completes the definition of the strategy $\sigma$.

 \begin{center}
 	\begin{table}
 		\centering
 			\begin{tabular}{c|c}
 				\multicolumn{2}{c}{$\bm_\omega(X)$ }\\
 				\hline
 				\hline
 				\multicolumn{2}{c}{\textcolor{white}{$GS_{\lambda}(A,k)$}} \\
 				$\1$ & $\2$ \\ 
 				\hline
 				$\textcolor{black}{A_{0}}$ & \\
 				&  $\textcolor{black}{\sigma(\langle A_{0}  \rangle)=\{B \in \mB: B \subset A_0\}} $\\
 				$\textcolor{black}{\mA_{1}}$ & \\
 				& $\textcolor{black}{\sigma(\langle A_{0}, \mA_{1} \rangle) = \bigcup_{A\in \mathcal{A}_{1}}\{B \in \mB: B \subset A   \}   \}}$ \\
 				\textcolor{white}{a}\\
 				\vdots & \vdots
 			\end{tabular}
 	\end{table}
 \end{center}
 Assume by way of contradiction that $\sigma$ is not winning, then there is a sequence $\seq{\mA_n: n \in \w}$ of plays for \1 that played against $\sigma$ will turn in a win for \1 - \emph{i.e.} $\bigcap_{n \in \w} \bigcup \mA_n = \emptyset$. 

 Now we will define a winner strategy $\rho$ for \1 in the $\bm(X)$ game. 
 
 Indeed, in the first inning, \textcolor{black}{\1 plays $\rho(\seq{}) = A_0$}  next \textcolor{black}{\2 responds $B_0 \subset A_0$} be any element from $\mB$, note that $B_0\in\sigma(\seq{\mA_{0}})$. In the second inning, in $\bmc(X)$, by definition, there is an $A_1 \in \mA_1$ such that $A_1 \subset B_{0}$. Then, in $\bm(X)$, \textcolor{black}{\1 plays $\rho(\seq{B_{0}}) = A_1$} next \textcolor{black}{\2 plays $B_{1}$} be any element from $\mB$ contained in $A_1$, note that $B_{1}\in \sigma(\seq{A_{0},\mA_{1}})$, and so on. This completes the definition of the strategy $\rho$.


\begin{center}
	\begin{table}[h!]
		\centering
		\begin{tabular}{c|c}
			\multicolumn{2}{c}{$\bm(X)$ }\\
			\hline
			\hline
			\multicolumn{2}{c}{\textcolor{white}{$GS_{\lambda}(A,k)$}} \\
			$\1$ & $\2$ \\ 
			\hline
			$\textcolor{black}{\rho(\seq{})  =A_{0}}$ & \\
			&  $\textcolor{black}{B_{0}} $\\
			$\textcolor{black}{\rho(\seq{B_{0}})=A_{1}}$ & \\
			& $\textcolor{black}{B_{1}}$ \\
			\textcolor{white}{a}\\
			\vdots & \vdots
		\end{tabular}
	\end{table}
\end{center}

Let us prove now that $\rho$ is winning for \1 in $\bm(X)$. So fix a play $$\seq{A_0, B_0, A_1, B_1, ..., A_n, B_n, ...}$$ following $\rho$ (with notation similar to what we presented in the construction of $\rho$). We need to show that $\bigcap_{n \in \w} B_n = \emptyset$. In fact, note that for each $n\in\omega$, $B_{n}\subset \bigcup\mA_{n}$, then $\bigcap_{n\in\omega}B_{n}\subset\bigcap_{n \in \w}\bigcup\mA_{n}=\emptyset$. Thus, \1 wins the game.
\end{proof}


\begin{corollary}
	For a space $X$ with locally countable $\pi$-base, the game $\bmc(X)$ is determined.
\end{corollary}

\section{A relation between $\bmfin$ and $\bmc$.}

In this part we will see that, assuming the Continuum Hypothesis, there is a space in which the games $\bmfin$ and $\bmc$ are different. We begin with a compilation of some known facts about subsets of the real line and we present some technical lemmas that will help us to conclude our objective.

\begin{lemma}
  Let $\mB$ be an infinite countable base for the topology on $\RR$. Then the set of strategies for \2 in the game $\bmfin(\RR)$ has cardinality at most $\mathfrak c$.
\end{lemma}

\begin{proof}
Remember that a strategy for \2 on $\bmfin(\RR)$, where both players only play open sets from $\mathcal{B}$, is a function $$\sigma:S\subseteq \mB^{<\omega}\to \mB^{<\omega}$$
As $|\mB^{<\omega}|=|\mB|=\omega$, we have that $|S|\leq \omega$, on the other hand, note that $ \{ \seq{B} : B\in \mB  \} \subseteq S$, so $|S|=\omega$. Finally,  $|\{ \sigma \in \mB^{S}: \sigma\hspace{0.1cm}\mbox{is strategy for}\hspace{0.1cm}\2  \}|\leq |\mB|^{|S|}=\omega^{\omega}=\mathfrak{c}$.

\end{proof}

\begin{definition}
Let $A\subseteq \RR$, we say that $A$ is 	
\begin{itemize}
	\item \textbf{nowhere dense}, if $\text{int}\hspace{0.05cm}(\overline{A})=\emptyset$. 
	\item \textbf{meager}, if $A$ is a countable union of nowhere dense sets.
	\item \textbf{residual}, if A is the complement of a meager set.
\end{itemize}
	
\end{definition}

\begin{theorem}[Oxtoby]
Let $A\subseteq \RR$, then $A$ is residual if and only if $A$ contains a dense $G_{\delta}$ subset of $\RR$.
\end{theorem}

Remember the following cardinalities of the following families of subsets of the real line:

\begin{itemize}
	
	
	
	\item  $|  \{M\subseteq \RR : M \hspace{0.1cm}\text{is meager}\}|=2^{\mathfrak{c}}$.
	
	\item $\mathcal{G}_{\delta}=\{G\subseteq \RR : G \hspace{0.1cm}\text{is a}\hspace{0.1cm} G_{\delta}\text{-set} \}$, then $|\mathcal{G}_{\delta}|=\mathfrak{c}$.
	\item $\mathcal{F}_{\sigma}=\{F\subseteq \RR : F \hspace{0.1cm}\text{is a}\hspace{0.1cm} F_{\sigma}\text{-set} \}$, then $|\mathcal{F}_{\delta}|=\mathfrak{c}$.
	\item $\mathcal{D}\mathcal{G}_{\delta}=\{D\subseteq \RR : D \hspace{0.1cm}\text{is a dense}\hspace{0.1cm} G_{\delta}\text{-set} \}$, then $|\mathcal{D}\mathcal{G}_{\delta}|=\mathfrak{c}$, because $\{\mathbb{R}\setminus\{x\}:x\in\RR   \}\subseteq \mathcal{D}\mathcal{G}_{\delta}$.
    \item Let $\mB=\{B_{n} :n\in\omega  \}$ be a basis for the topology on $\RR$, then the set $\{G\in \mathcal{G}_{\delta} \mid \exists n\in\omega : G\subseteq B_{n} \subseteq \overline{G}  \}$ has cardinality $\mathfrak{c}$, because for each $n\in \omega$, $\{B_{n}\setminus\{b\} : b\in B_{n}\} \subseteq \{G\in \mathcal{G}_{\delta} \mid \exists n\in\omega : G\subseteq B_{n} \subseteq \overline{G}  \}   $.
\end{itemize}


\begin{lemma}\label{easytec}
	Let $\mB=\{B_{n} :n\in\omega  \}$ be a basis for the topology on $\RR$. For each $G\in \mathcal{D}\mathcal{G}_{\delta}$ and $m\in\omega$, we have $G\cap B_{m} \in \{G\in \mathcal{G}_{\delta} \mid \exists n\in\omega : G\subseteq B_{n} \subseteq \overline{G}  \}$.
\end{lemma}
\begin{proof}
	Note that $G\cap B_{m}  \in \mathcal{G}_{\delta}$ and $G\cap B_{m}\subseteq B_{m}$ . Also, $B_{m}\subseteq \overline{G\cap B_{m}}$. In fact, otherwise, there is a $z\in B_{m}$ such that $z\not\in \overline{G\cap B_{m}}$, so there is an $n\in \omega$ with $z\in B_{n}$ such that   $B_{n}\cap B_{m}\cap G=\emptyset$, contradiction, because $G$ is dense.

\end{proof}

\begin{proposition}\label{hrusak}
The following properties of a topological space $X$ are equivalent:
	
\begin{enumerate}
	\item [(1)] $X$ is a Baire space.
 	\item [(2)] For every countable closed cover $\{H_{n} : n\in\omega   \}$ of $X$, the set $\bigcup_{n\in\omega}\text{int}(H_{n}) $ is dense in $X$.
	\item [(3)] For every sequence $V_{0}, V_{1}, \cdots$ of open sets with the same closure $K$, we have $K=\overline{\bigcap_{n\in\omega}V_{n}}$.
	\item [(4)] Every meager $G_{\delta}$-set in $X$ is nowhere dense. 
	\item [(5)] Every meager set has empty interior.
	
\end{enumerate}		
\end{proposition}

Also, every topological space which has a dense Baire subspace is evidently a Baire
space. The converse is not true: for instance, the real line is a Baire space but the
subspace of rationals is not.	




\begin{lemma}\label{densoBairechar}
	Let $X\subseteq\RR$ be dense, then 
	$X$ is Baire if and only if $G\cap X$ is dense in $X$, for each $G\in\mathcal{D}\mathcal{G}_{\delta}$.
\end{lemma}

\begin{proof}
Suppose that $X$ is Baire and let $G=\bigcap_{n \in \omega}G_{n}$ be a dense $G_{\delta}$-set in $\RR$. Then for each $n\in\omega$, $G_{n}$ is an open dense set in $\RR$, therefore $G_{n}\cap X$ is an open dense set in $X$. As $X$ is Baire. Then $\bigcap_{n \in \omega}G_{n}\cap X=G\cap X$ is dense in $X$.

Now, let $\langle A_{n}:n\in\omega\rangle$ be a sequence of open dense sets in $X$. Then $A_{n}=A^{\prime}_{n}\cap X	$ with $A^{\prime}_{n}$ a non-empty open set in $\RR$ and, as $A_{n}$ is dense in $X$, $A^{\prime}_{n}$ is open and dense in $\RR$. Therefore $\bigcap_{n \in \omega} A^{\prime}_{n}$ is a dense $G_{\delta}$-set in $\RR$, then by hypothesis, $\bigcap_{n \in \omega}A_{n}=\bigcap_{n \in \omega}A^{\prime}_{n}\cap X$ is dense in $X$.

\end{proof}

\begin{definition}
	Given $\sigma$ a strategy for \2 in the $\bmfin(\RR)$ and $s = \seq{\mA_n: n \in \w}$ a sequence of finite collections of open sets, we define $\sigma * s$ as the set
	\[\bigcap_{n \in \w} \bigcup \sigma(\seq{\mA_0, ..., \mA_n}).\]
	
\end{definition}

We are assuming that if we use this notation, then $\sigma$ and the sequence are compatible, in the sense that the previous intersection is well defined.

\begin{lemma}\label{lematec}
		Let $\sigma$ be a strategy for $\2$ in the $\bmfin({\mathbb R})$, $X$ be a countable set of points in $\RR$ and $\mathcal{Y}$ be a countable family of nowhere dense sets in $\RR$. Then there exists a sequence $s$ of finite collections of basic open sets (which will be $\1$'s choices) such that $\sigma*s$ is a closed nowhere dense set disjoint from $X\cup \bigcup \mathcal{Y}$.
\end{lemma}

\begin{proof}
Let $X=\{x_{n}:n\in\omega \}$, $\mathcal{Y}=\{Y_{n}:n\in\omega   \}$ and fix $\mB=\{B_{n}: n\in\omega   \}$ be a basis for $\RR$.

Let $x_{0}\in X$, and note that $(\RR\setminus\{ x_{0} \}) \cap (\RR\setminus \overline{Y_{0}})\cap B_{0}$ is a non-empty open set. So consider $z_{0}\in (\RR\setminus\{ x_{0} \}) \cap (\RR\setminus \overline{Y_{0}})\cap B_{0}$ and note that $(\RR\setminus\{ z_{0} \})\cap (\RR\setminus\{ x_{0} \}) \cap (\RR\setminus \overline{Y_{0}})\cap B_{0}$ is a non-empty open set. Then choose $\mA_0=\{A_0\}$ such that $\overline{A_{0}}\subseteq  (\RR\setminus\{ z_{0} \})\cap (\RR\setminus\{ x_{0} \}) \cap (\RR\setminus \overline{Y_{0}})\cap B_{0}$ and note that $x_{0}, z_{0}\not\in A_{0}$.

Now, in the game $\bmfin(\RR)$, we have that in the first inning $\1$ can play $A_{0}$. Note that $B_{0}\not\subseteq A_{0}$ and $(\{x_{0}\}\cup Y_{0}) \cap A_{0}=\emptyset$. Next $\2$ responds $\sigma(\langle A_{0}  \rangle)=\{B^{0}_{0}, \cdots, B^{0}_{n_{0}}\}$.

Let $x_{1}\in X$ and $j\in\{0,\cdots, n_{0} \}$, consider $B^{0}_{j}$ and note that $(\RR\setminus \{x_{0}  \})\cap (\RR\setminus\overline{Y_{0}} ) \cap (\RR\setminus \{x_{1}  \}) \cap  (\RR\setminus\overline{Y_{1}} ) \cap B_{1}$ is a non-empty open set. So let $z_{1} \in (\RR\setminus \{x_{0}  \})\cap (\RR\setminus\overline{Y_{0}} ) \cap (\RR\setminus \{x_{1}  \}) \cap  (\RR\setminus\overline{Y_{1}} ) \cap B_{1}$ and note that $(\RR\setminus \{z_{1} \}) \cap (\RR\setminus \{x_{0}  \})\cap (\RR\setminus\overline{Y_{0}} ) \cap (\RR\setminus \{x_{1}  \}) \cap  (\RR\setminus\overline{Y_{1}} ) \cap B^{0}_{j}$ is a non-empty open set. Then choose $A^{1}_{j}\in\mB$
such that $\overline{A^{1}_{j}}  \subseteq (\RR\setminus \{z_{1} \}) \cap (\RR\setminus \{x_{0}  \})\cap (\RR\setminus\overline{Y_{0}} ) \cap (\RR\setminus \{x_{1}  \}) \cap  (\RR\setminus\overline{Y_{1}} ) \cap B^{0}_{j}$ and note that for each $j\in\{0, \cdots, n_{0} \} $, $z_{1}\not\in A^{1}_{j}$.

Then in the second inning, $\1$ plays $\mA_{1}=\{A^{1}_{0}, \cdots, A^{1}_{n_{0}}\}$ and note that $B_{1}\not\subseteq \bigcup \mathcal{A}_{1}$ and $(\{x_{0}, x_{1}\} \cup Y_{0}\cup Y_{1}) \cap \bigcup \mA_{1}=\emptyset$. Next $\2$ responds $\sigma(\langle A_{0}, \mA_{1}\rangle) = \{ B^{1}_{0}, \cdots, B^{1}_{n_{1}}\} $.

In general, for each inning $m\in\omega$, suppose $\2$ played $\{B^{m-2}_{0}, \cdots, B^{m-2}_{n_{m}}\}$ in the last inning. Let $x_{m}\in X$ and $j\in\{0, \cdots, n_{m}  \}$, consider $B^{m-2}_{j}$ and note that,      
$$ \bigcap_{i=0}^{j} (\RR\setminus \{x_{i}   \}    )  \cap  \bigcap_{i=0}^{j} (\RR\setminus \overline{Y_{i}}) \cap B_{j}$$ is a non-empty open set. So put $z_{m}\in \bigcap_{i=0}^{j} (\RR\setminus \{x_{i}   \}    )  \cap  \bigcap_{i=0}^{j} (\RR\setminus \overline{Y_{i}}) \cap B_{j}      $ and note that $(\RR\setminus \{z_{m}\})\cap \bigcap_{i=0}^{j} (\RR\setminus \{x_{i}   \}    )  \cap  \bigcap_{i=0}^{j} (\RR\setminus \overline{Y_{i}}) \cap B^{m-2}_{j}$ is a non-empty open set. Then fix $A^{m-1}_{j}\in\mB$ such that 
$$\overline{A^{m-1}_{j}}\subseteq(\RR\setminus \{z_{m}\})\cap \bigcap_{i=0}^{j} (\RR\setminus \{x_{i}   \}    )  \cap  \bigcap_{i=0}^{j} (\RR\setminus \overline{Y_{i}}) \cap B^{m-1}_{j}      $$
Then in the inning $m$, $\1$ plays $\mA_{m-1}=\{A^{m-1}_{0}, \cdots, A^{m-1}_{n_{m}}\}$ and note that $$B_{m-1}\not\subseteq \bigcup \mA_{m-1} \hspace{0.1cm}\text{and}  \hspace{0.1cm}(\{x_{0}, \cdots, x_{m-1}\}\cup  \bigcup_{i=0}^{m-1} Y_{i})\cap \bigcup \mA_{m-1} =\emptyset  $$

For each $n\in\omega$, denote $\tilde{\mA}_{n}=\{\overline{A} : A\in \mA_{n}\}$, note that $\bigcup \tilde{\mA}_{n}$ is a closed set, because $\mA_{n}$ is finite. 

\begin{claim}
Let $s=\langle \mA_{n} : n\in \omega\rangle$. Then $\sigma*s$ is a closed set with empty interior.
\end{claim}
	\begin{proof}
	\textcolor{white}{text}	
	\begin{enumerate}
		\item $\sigma*s$ is closed.
		
		In fact, we will show that $\sigma*s=\bigcap_{n\in\omega}\bigcup \tilde{\mathcal{A}}_{n}$.
		
		$(\subseteq)$ Let $z\in \sigma*s = \bigcup_{n\in\omega} \sigma(\langle A_{0}, \cdots, \mA_{n}\rangle)$. Then there is a $B^{n}_{j}\in \sigma(\langle A_{0}, \cdots, \mA_{n}\rangle)$ such that $z\in B^{n}_{j}\subseteq \bigcup \mA_{n} \subseteq \bigcup \tilde{\mA}_{n}$.

		$(\supseteq)$ Let $z\in \bigcap_{n\in\omega}\bigcup \tilde{\mathcal{A}}_{n}$. Then for every $n\in\omega$, we have $z\in \overline{A^{n+1}_{j}}\subseteq B^{n}_{j} \subseteq \bigcup \sigma(\langle A_{0}, \cdots, \mA_{n} \rangle)$.

		\item $\sigma*s$ has empty interior.
		
		Suppose otherwise, then there exists $n_{0}\in \omega$ such that  $B_{n_0}\subseteq \sigma*s = \bigcap_{n\in\omega}\bigcup \tilde{\mathcal{A}}_{n}\subseteq \bigcap_{n\in\omega}\bigcup \mathcal{A}_{n}\subseteq \bigcup \mA_{n_{0}}$, contradiction by our construction.
	\end{enumerate}

\end{proof}

\begin{claim}
	$(\sigma*s) \cap (X\cup \bigcup\mathcal{Y} ) =\emptyset$
\end{claim}
\begin{proof}
Suppose otherwise, there exists $$z\in (\sigma*s) \cap (X\cup \bigcup\mathcal{Y}) = \bigcap_{n \in \w} \bigcup \sigma(\seq{\mA_0, ..., \mA_n}) \cap (X\cup \bigcup\mathcal{Y})$$ Then there exists $n\in\omega$ such that $$z\in (\{x_{0}, \cdots, x_{n}\}  \cup \bigcup_{j=0}^{n}Y_{j}  ) \cap \bigcup \mA_{n},$$
but by our construction $ (\{x_{0}, \cdots, x_{n}\}  \cup \bigcup_{j=0}^{n}Y_{j}  ) \cap \bigcup \mA_{n}=\emptyset$, contradiction.

\end{proof}

Therefore $(\sigma*s)$ is a closed nowhere dense set disjoint from $X\cup \bigcup \mathcal{Y}$.

\end{proof}

\newpage

The following is immediate 

\begin{lemma}\label{2inducesws}
	Let $Y$ be a dense subspace of $X$ in which $\2$ has a winning strategy in $\bmfin(Y)$, then $\2$ has a winning strategy in $\bmfin(X)$.
\end{lemma}

\begin{proposition}\label{propfinal}
  (CH) There is a Baire subspace $X$ of $\RR$ such that \2 has no winning strategy for $\bmfin(X)$. 
\end{proposition}


\begin{proof}
Fix $\mB=\{B_{n}:n\in\omega\}$ an infinite basis for the topology of $\RR$, 
let $\{G\in \mathcal{G}_{\delta} \mid \exists n\in\omega : G\subseteq B_{n} \subseteq \overline{G}  \} = \{G_{\xi}:\xi\in\omega_{1}\}$, and fix an enumeration $\{\sigma_{\xi} : \xi<\omega_{1} \}$ of $\2$'s strategies in the game $\bmfin(\RR)$.

By transfinite induction on $\xi < \omega_{1}$ we will construct a sequence $\{x_{\xi} : \xi < \omega_{1}  \}$ of points of $\RR$ such that for every $\xi < \omega_{1}$

\begin{itemize}
\item [(a)] $x_{\xi}\in G_{\xi}$
\item [(b)] there exists a sequence of finite collections of open sets $s_{\xi}$ such that
$$(\sigma_{\xi}*s_{\xi})\cap \left(\{x_{\beta}:\beta\leq\xi   \} \cup \bigcup_{\beta<\xi}(\sigma_{\beta}*s_{\beta})    \right) =\emptyset$$

\end{itemize}

Consider $x_{0}\in G_{0}$ and $\sigma_{0}$. By Lemma \ref{lematec}, there exists $s_{0}$ such that $\sigma_{0}*s_{0}$ is a closed nowhere dense set and $(\sigma_{0}*s_{0})\cap \{x_{0}\}=\emptyset$. Set $x_{1}\in G_{1} \setminus (\{x_{0} \} \cup  (\sigma_{0}*s_{0})   )  $, by Lemma \ref{lematec}, there exists $s_{1}$ such that $\sigma_{1}*s_{1}$ is a closed nowhere dense set and $(\sigma_{1}*s_{1})\cap (\{x_{0}, x_{1}\} \cup (\sigma_{0}*s_{0})  )  =\emptyset$. 

Now suppose that we have built $\{x_{\beta} : \beta<\alpha\}$ with $\alpha<\omega_{1}$, note that $$  \bigcap_{\beta<\alpha}\RR\setminus (\sigma_{\beta}*s_{\beta})$$
is a dense $G_{\delta}$-set. 
We claim that $G_{\alpha} \setminus \left( \{x_{\beta} : \beta<\alpha     \} \cup \bigcup_{\beta<\alpha} (\sigma_{\beta}*s_{\beta})       \right) \not=\emptyset$. In fact, suppose otherwise, $G_{\alpha}\subseteq  
\{x_{\beta} : \beta<\alpha     \} \cup \bigcup_{\beta<\alpha} (\sigma_{\beta}*s_{\beta})$, so $G_{\alpha}$ is a meager $G_{\delta}$-set in $\RR$, then, by Proposition \ref{hrusak}, $G_{\alpha}$ is nowhere dense, contradiction, because $G_{\alpha}\in \{G\in \mathcal{G}_{\delta} \mid \exists n\in\omega : G\subseteq B_{n} \subseteq \overline{G}  \}$. Therefore, choose $x_{\alpha} \in 
G_{\alpha} \setminus \left( \{x_{\beta} : \beta<\alpha     \} \cup \bigcup_{\beta<\alpha} (\sigma_{\beta}*s_{\beta})       \right)$.
By Lemma \ref{lematec}, there exists $s_{\alpha}$ such that $\sigma_{\alpha}*s_{\alpha}$ is a closed nowhere dense set and $$(\sigma_{\alpha}*s_{\alpha})\cap \left(\{x_{\beta}:\beta\leq\alpha   \} \cup \bigcup_{\beta<\alpha}(\sigma_{\beta}*s_{\beta})    \right) =\emptyset$$ 

Finally denote $X=\{x_{\xi} : \xi\in\omega_{1}\}$, note that $X$ is a dense set. Indeed, let $m\in\omega$ and consider $b\in B_{m}$, so $(B_{m}\setminus \{ b \}) \in \{G\in \mathcal{G}_{\delta} \mid \exists n\in\omega : G\subseteq B_{n} \subseteq \overline{G}  \} = \{G_{\xi}:\xi\in\omega_{1}\}$. Then there is a $\xi_{m}\in\omega_{1}$ such that $B_{m}\setminus\{b\} = G_{\xi_{m}}$, by construction $x_{\xi_{m}} \in G_{\xi_{m}}\cap X \subseteq B_{m}\cap X$.

\begin{claim}
		$X$ is a Baire space. In particular $\1$ has no winnnig strategy in $\bmfin(X)$. 
\end{claim}	

\begin{proof}
Suppose otherwise, then by Lemma \ref{densoBairechar}, there exists $G\in\mathcal{D}\mathcal{G}_{\delta}$ such that $G\cap X$ is not dense in $X$. So there is an $m\in\omega$ such that $(B_{m}\cap X) \cap (G\cap X) = B_{m} \cap (G\cap X) =\emptyset$. Then, by Lemma \ref{easytec}, there is a $\xi<\omega_{1}$ such that $G_{\xi}= B_{m}\cap G$, but by construction $x_{\xi}\in G_{\xi}\cap X = B_{m}\cap G \cap X =\emptyset$, contradiction.

\end{proof}

\begin{claim}
		$\2$ has no winning strategy in $\bmfin(X)$.
\end{claim}

\begin{proof}
Suppose otherwise, that is, $\2$ has a winning strategy $\tilde{\sigma}_{_{X}}$ in $\bmfin(X)$. As $X$ is dense, by Lemma \ref{2inducesws}, $\tilde{\sigma}_{_{X}}$ induces a winning strategy $\tilde{\sigma}$ for $\2$ in $\bmfin(\RR)$. Also, $\tilde{\sigma}$ satisfies that for each sequence  $s=\langle\mA_{n} : n\in\omega\rangle$ of $\1$'s moves in $\bmfin(\RR)$,
$$\tilde{\sigma}_{_{X}}*s^{\prime}  =  \bigcap_{n \in \omega}\bigcup \tilde{\sigma}_{_{X}}(\langle \mA^{X}_{0}, \cdots, \mA^{X}_{n} \rangle) \subseteq  \bigcap_{n \in \omega} \bigcup \tilde{\sigma}(\langle \mA_{0}, \cdots, \mA_{n} \rangle) \cap X  = (\tilde{\sigma}*s )\cap X   $$
where $s^{\prime}=\langle  \mA_{n}\cap X  : n\in\omega  \rangle $ with $\mA_{n}\cap X = \{A\cap X : A\in \mA_{n} \}$. In particular, there is a $\xi<\omega_{1}$ such that $\tilde{\sigma}=\sigma_{\xi}$. 

Let $s_{\xi}=\langle \mA^{\xi}_{n} : n\in \omega \rangle$ of our construction. 

\begin{claim}
	$(\sigma_{\xi} * s_{\xi})\cap X=\emptyset$.
\end{claim}

\begin{proof}
In fact, suppose otherwise, there is an $\alpha<\omega_{1}$ such that $x_{\alpha} \in (\sigma_{\xi}*s_{\xi})$. We have the following cases:
\begin{itemize}
	\item $\alpha\leq \xi$: by part (b),
	$(\sigma_{\xi}*s_{\xi})\cap \{x_{\beta}:\beta\leq\xi \}=\emptyset$, contradiction. 
	\item $\alpha>\xi$: by construction, $x_{\alpha} \in 
	G_{\alpha} \setminus \left( \{x_{\beta} : \beta<\alpha     \} \cup \bigcup_{\beta<\alpha} (\sigma_{\beta}*s_{\beta})       \right)$, contradiction.
\end{itemize}

\end{proof}

So $\emptyset\not= \tilde{\sigma}_{_{X}}*s^{\prime}_{\xi}  \subseteq    (\sigma_{\xi} * s_{\xi})\cap X =\emptyset$, contradiction. Then, $\2$ has no winning strategy in $\bmfin(X)$.

\end{proof}

Therefore $X$ is the desired space.		
\end{proof}

\begin{corollary}
		(CH) There exists a subspace $X$ of the real line in which the game $\bmfin(X)$ is not determined.
\end{corollary}

\begin{corollary}\label{chcor}
		(CH) There is a productively Baire space such that \2 has no winning strategy for the $\bmfin$ game.
\end{corollary}

\begin{proof}
	Let $X$ as in Proposition \ref{propfinal}, then by Corollary \ref{cpbpb}, $X$ is productively Baire.
\end{proof}

\begin{question}
	Corollary \ref{chcor} holds if we change $\bmfin$ to $\bmc$?
\end{question}

\begin{question}
Is there any relationship between the game $\bmc$ and the productively Baire spaces?	
\end{question}

\begin{corollary}
(CH) There exists a space $X$ such that Bob has a winning strategy in $\bm_\omega$ but $\2$ has no winning strategy in $\bmfin$.
\end{corollary}

\end{document}